\documentclass[preprint,3p,10pt]{elsarticle}

\usepackage{amsmath,amsfonts,amssymb,mathtools}
\usepackage{bbm}
\usepackage{graphicx}
\usepackage{exscale}
\usepackage{hhline}
\usepackage{array}
\usepackage{multirow}

\newtheorem{Lm}{Lemma}

\newtheorem{exmp}{Example}

\DeclareMathAlphabet\mathbfcal{OMS}{cmsy}{b}{n}

\usepackage{clrscode}

\begin{document}

\begin{frontmatter}
\title{Square-root filtering via covariance SVD factors in the accurate continuous-discrete extended-cubature Kalman filter}

\author[CEMAT]{Maria V.~Kulikova\corref{cor}} \ead{maria.kulikova@ist.utl.pt} \cortext[cor]{Corresponding
author.}
\author[CEMAT]{Gennady Yu.~Kulikov} \ead{gkulikov@math.ist.utl.pt}

\address[CEMAT]{CEMAT (Center for Computational and Stochastic Mathematics), Instituto Superior T\'{e}cnico, Universidade de Lisboa, \\ Av. Rovisco Pais 1,  1049-001 Lisboa, Portugal}

\begin{abstract}
This paper continues our research devoted to an accurate nonlinear Bayesian filters' design. Our solution implies numerical methods for solving ordinary differential equations (ODE) when propagating the mean and error covariance of the dynamic state. The key idea is that an accurate implementation strategy implies the methods with a discretization error control involved. This means that the filters' moment differential equations are to be solved accurately, i.e. with negligible error. In this paper, we explore the continuous-discrete extended-cubature Kalman filter that is a hybrid method between Extended and Cubature Kalman filters (CKF). Motivated by recent results obtained for the continuous-discrete CKF in Bayesian filtering realm, we propose the numerically stable (to roundoff) square-root approach within a singular value decomposition (SVD) for the hybrid filter. The new method is extensively tested on a few application examples including  stiff systems.
\end{abstract}


\end{frontmatter}


\section{Introduction}

There is a growing body of literature that recognizes an importance of effective variable-stepsize ODE solvers with automatic discretization error control facilities in nonlinear Bayesian filtering realm. They are utilized for solving the so-called {\it moment differential equations} derived for each filtering framework and allow for an accurate propagation step of various filtering methods. The accuracy and good estimation quality are achieved due to an automatic discretization error control involved. Our solution is grounded in the variable-stepsize Nested Implicit Runge-Kutta (NIRK) pairs with built-in local and global error controls proposed in~\cite{2013:Kulikov:IMA}. Previously, we have suggested the adaptive NIRK-based hybrid method between Extended and Cubature Kalman filters (EKF-CKF) in~\cite{2017:Kulikov:ANM}. Additionally, we have derived the robust square-root implementations for the mentioned filtering framework within the Cholesky decomposition of the filters' covariance matrices. Our current research is motivated by the most recent result obtained for the Cubature Kalman filter (CKF), which is the singular value decomposition (SVD) approach developed in~\cite{2020:Automatica:Kulikova}. The goal of this paper is to derive the SVD-based square-root method for the accurate EKF-CKF suggested in~\cite{2017:Kulikov:ANM}. This expands the traditionally used Cholesky-based framework on the SVD square-rooting solution.

The continuous-discrete CKF methods have been proposed for estimating a hidden dynamic state of nonlinear stochastic systems in~\cite{2010:Haykin,2018:Haykin}. The recent SVD-based square-root CKF algorithms developed in~\cite{2020:Automatica:Kulikova} belong to the class of {\it factored-form} (square-root) implementations. They yield the improved numerical stability with respect to roundoff errors by ensuring the positive (semi-) definiteness and symmetric form of the computed error covariance matrix~\cite{2015:Grewal:book,2010:Grewal}. The earlier research on the stable CKF implementation framework has brought the traditional square-root methods via Cholesky decomposition for both the covariance~\cite{2009:Haykin,2010:Haykin} and information filtering~\cite{2017:Arasaratnam,2013:Chandra}. Compared to the Cholesky solution cited, the novel SVD-type square-rooting idea  provide users with an extra information about the filters matrix structures involved and, hence, some valuable insights about the underlying estimation
process. Indeed, the SVD-based implementations are capable to oversee the eigenfactors (i.e. the eigenvalues and eigenvectors) while estimation process, i.e. the singularities may be revealed as soon as they appear; see the discussion of some other benefits in~\cite{1988:Oshman,1983:Ham} and the SVD-based implementation methods for the classical Kalman filter in~\cite{1986:Oshman,1988:Oshman,1992:Wang,2017:Kulikova:IET,2017:Tsyganova:IEEE}.

The discussed continuous-discrete CKF estimators in~\cite{2010:Haykin,2018:Haykin,2020:Automatica:Kulikova} are designed by using the Euler-Maruyama method and It\^{o}-Taylor expansion for discretizing the underlying stochastic differential equations (SDEs) of the system at hand. The following drawbacks of this continuous-discrete filtering approach are worth to be mentioned. First, SDEs solvers are fixed step-size numerical integration schemes~\cite{1999:Kloeden:book} and, hence, the resulting CKF methods require a special manual tuning from the users prior to filtering for generating the mesh. For instance, this property  makes the novel SVD-based CKF estimators in~\cite{2020:Automatica:Kulikova} incapable to process the missing measurement cases, accurately. They also do not provide a good estimation quality when solving estimation problems with irregular and/or long sampling intervals. Next, any SDEs solver implies no opportunity to control and/or bound the occurred discretization error due to its stochastic nature. An alternative implementation framework assumes the derivation of the related filters' moment differential equations and then an utilization of the numerical methods derived for solving ordinary differential equations (ODEs); e.g., see the discussion in~\cite{2018:Haykin,2012:Frogerais,2014:Kulikov:IEEE,2016:Kulikov:SISCI}. This approach leads to a more accurate implementation framework because of the discretization error control techniques involved. The moment differential equations are derived for the continuous-discrete unscented Kalman filter (UKF) and CKF estimators in~\cite{2007:Sarkka,2012:Sarkka} as well as for the extended Kalman filter (EKF) in~\cite{1970:Jazwinski:book}. It is worth noting here that the continuous-discrete UKF methods with the NIRK pairs and built-in local and global error controls have been recently developed in conventional and Cholesky-based square-root forms in~\cite{KuKu17SP} and~\cite{KuKu20aANM}, respectively. Alternative advanced ODE solvers with automatic stepsize selection and error control facilities, especially those grounded in Runge-Kutta, general linear and peer methods published recently in~\cite{Bu08,Co18,GoHe10,GoHe12,Ja09,KuWe10SISCI,KuWe11CAM,KuWe15SISCI,ScWe04,ScWe05BIT,KuWe20bANM,WeKu14CAM,WeKu17CAM,WeKu12ANM,WeSc09,abdi2021global}, can contribute to nonlinear Bayesian filtering realm for treating complicated state estimation scenarios.

The goal of this paper is to design the numerically stable (to roundoff) SVD factorization-based method for the NIRK-based mixed-type EKF-CKF estimator in~\cite{2017:Kulikov:ANM}. The mentioned hybrid filter provides a good balance in trading between estimation accuracy and computational demand because the EKF moment differential equations are not coupled and can be solved separately~\cite{2017:Kulikov:IEEE:mix}. This is not a case for the UKF and CKF estimators; e.g., see the discussion in~\cite{2017:Kulikov:ANM,2008:Mazzoni}. Thus, the key feature of the mixed-type method is to apply the advanced ODE numerical integration schemes with the discretization error control involved for solving the filter's mean equations, only. The mesh is generated automatically according to the utilized stepsize selection mechanization based on the chosen error control in order to keep the discretization error less than the pre-defined tolerance given by users. The related covariance equations are solved separately and without any monitoring technique, but on the same adaptive mesh generated. This mixed-type approach simplifies the practical implementations meanwhile it still provides an accurate mean propagation with the reduced discretization error according to the given tolerance value in automatic mode. However, an additional special task is to ensure the theoretical properties of the error covariance matrix computed. In this paper, we suggest the SVD-based solution via the filter covariance eigenfactors. All of this yields an {\it accurate} and {\it robust} (with respect to roundoff) mixed-type continuous-discrete EKF-CKF filtering. Finally, the novel method is fairly tested on a few application examples including the so-called stiff systems. The results of numerical experiments clearly indicate a superior performance of the new continuous-discrete estimator over the previously suggested CKF algorithms cited in this paper.

The paper is organized as follows. Section~\ref{problem:statement} gives a brief overview of the continuous-discrete NIRK-based hybrid EKF-CKF estimator, which is under examination in this paper. Section~\ref{main:result} suggests the new SVD-based square-root solution for the EKF-CKF. The results of numerical experiments are discussed in Section~\ref{numerical:experiments}. Finally, some open problems are mentioned in Section~\ref{conclusion} as well as their possible solutions, which might be developed in future.

\section{Continuous-discrete NIRK-based Hybrid Extended-Cubature Kalman Filter} \label{problem:statement}

Consider continuous-discrete stochastic system of the form
\begin{align}
dx(t) & = f\bigl(t,x(t)\bigr)dt+Gd\beta(t), \quad t>0,  \label{eq1.1} \\
z_k   & =  h(k,x(t_{k}))+v_k, \quad k =1,2,\ldots \label{eq1.2}
\end{align}
where  $x(t)$ is the $n$-dimensional unknown state vector to be estimated and  $f:\mathbb R\times\mathbb
R^{n}\to\mathbb R^{n} $ is the time-variant drift function. The process uncertainty is modelled by the additive noise term where
$G \in \mathbb R^{n\times q}$ is the time-invariant diffusion matrix and $\beta(t)$ is the $q$-dimensional Brownian motion whose increment $d\beta(t)$ is Gaussian white process independent of $x(t)$ and has the covariance $Q\,dt>0$. Finally, the $m$-dimensional measurement vector $z_k = z(t_{k})$ comes at some discrete-time points $t_k$ with the sampling rate (sampling period) $\Delta_k=t_{k}-t_{k-1}$. The measurement noise term $v_k$ is assumed to be a white Gaussian noise with the zero mean and known covariance $R_k>0$, $R \in \mathbb R^{m\times m}$. Finally, the initial state $x(t_0)$ and the noise processes are assumed to be statistically independent, and $x(t_0) \sim {\mathcal N}(\bar x_0,\Pi_0)$, $\Pi_0 > 0$.

Following the CKF estimation methodology proposed in~\cite{2009:Haykin,2010:Haykin}, the third-degree spherical-radial cubature rule is utilized for computing the involved $n$-dimensional Gaussian-weighted integrals. For that, the cubature nodes (vectors) are defined as follows:
\begin{align}
 {\cal X}_{i}& = S_{P} \xi_i+\hat x, \quad i = 1, \ldots 2n,& \mbox{with} && \xi_j = \sqrt{n} e_j, \mbox{ and } \xi_{n+j}=-\sqrt{n} e_j, \quad j = 1, \ldots, n \label{cub:points}
\end{align}
where $e_j$ denotes the $j$-th unit coordinate vector in $\mathbb R^n$ and $n$ is the dimension of the state vector to be estimated.
Additionally, the term $\hat x$ is the estimate of the state vector and $S_{P}$ stands for a square-root factor of error covariance matrix $P$, i.e. the matrix is factorized as follows: $P = S_{P}S_{P}^{\top}$. Following~\cite{2009:Haykin,2010:Haykin,2018:Haykin}, the Cholesky decomposition is traditionally utilized for determining the matrix square-root factor, i.e. $P = P^{1/2}P^{{\top}/2}$ where $S_{P}:=P^{1/2}$ is a lower or upper triangular matrix with positive diagonal entries. However, the matrix square-root might be defined in other ways. In this paper, we apply the SVD spectral factorization and, hence, the SVD-based matrix square-root is defined as follows. For a symmetric positive-definite matrix, the factorization yields $P = Q_{P}D_{P}Q_{P}^{\top}$ where $Q_{P}$ is an orthogonal matrix and the diagonal part $D_{P}$ contains the singular values of $P$. Thus, the matrix square-root utilized in~\eqref{cub:points} might be defined by the formula $S_{P}:=Q_{P}D_{P}^{1/2}$. It is worth noting here that the Cholesky-based square-root factor is always a triangular matrix, meanwhile the SVD-based square-root factor is a full matrix and might be rectangular; e.g., when a number of singular values is $r < n$.

The CKF measurement update step for processing the measurement data $z_k$ at time instance $t_k$ can be written in a simple matrix-vector form as follows: given the predicted estimate $\hat x_{k|k-1}$ and covariance $P_{k|k-1}$, perform the following steps
\begin{align}
&\mbox{Generate the CKF nodes}: & {\cal X}_{i,k|k-1}& =S_{P_{k|k-1}}\xi_i+\hat x_{k|k-1}, \quad i=1,\ldots, 2n,  \label{start:cub}\\
&\mbox{Propagate the vectors}: & {\cal Z}_{i,k|k-1}& =h\bigl(k,{\cal X}_{i,k|k-1}\bigr), \label{propagateH} \\
&\mbox{Collect by columns}: & {\cal Z}_{k|k-1} &=\bigl[{\cal Z}_{1,k|k-1},\ldots,{\cal Z}_{2n,k|k-1} \bigr], \;{\cal X}_{k|k-1}=\bigl[{\cal X}_{1,k|k-1},\ldots,{\cal X}_{2n,k|k-1}\bigr],\label{matZ} \\
&\mbox{Compute the vector}: & \hat z_{k|k-1}  &=\frac{1}{2n}{\cal Z}_{k|k-1} {\mathbf 1}_{2n}, \\
&\mbox{Define the matrices}: & {\mathbb X}_{k|k-1} &=\frac{1}{\sqrt{2n}}\bigl({\cal X}_{k|k-1}\!-{\mathbf
1}_{2n}^{\top}\otimes \hat x_{k|k-1}\bigr), {\mathbb Z}_{k|k-1} =\frac{1}{\sqrt{2n}}\bigl({\cal Z}_{k|k-1}\!-{\mathbf
1}_{2n}^{\top}\otimes \hat z_{k|k-1}\bigr) \label{end:cub}
\end{align}
where ${\mathbf 1}_{2n}$ is the unitary column of size $2n$ and $I_{2n}$ is the identity matrix of that size, the symbol $\otimes$ is the Kronecker tensor product. Having defined the centered matrices ${\mathbb X}_{k|k-1} \in {\mathbb R}^{n\times 2n}$ and ${\mathbb Z}_{k|k-1} \in {\mathbb R}^{m\times 2n}$, compute the filtered error covariance matrix and the state estimate in a simple and elegant way as follows:
\begin{align}
&\mbox{Compute the residual covariance}: & R_{e,k} & ={\mathbb Z}_{k|k-1}{\mathbb Z}_{k|k-1}^{\top}+R_k, \label{ckf:rek} \\
&\mbox{Find the cross-covariance}: & P_{xz,k} & ={\mathbb X}_{k|k-1}{\mathbb Z}_{k|k-1}^{\top}, \label{cross-cov-ckf} \\
&\mbox{Calculate the cubature gain}: & {\mathbb K}_{k} &=P_{xz,k}R_{e,k}^{-1}, \label{ckf:gain} \\
&\mbox{Find the filtered estimate}: & \hat x_{k|k} &=\hat x_{k|k-1}+{\mathbb K}_k(z_k-\hat z_{k|k-1}), \label{ckf:f:x} \\
&\mbox{Find the filter covariance}: & P_{k|k} &=P_{k|k-1} - {\mathbb K}_k R_{e,k} {\mathbb K}_k^{\top}. \label{ckf:f:p}
\end{align}

The mixed-type EKF-CKF estimator under examination is coupled with the {\it continuous-discrete} EKF implementation framework for performing the time update step in an accurate way. More precisely, consider the EKF moment differential equations in~\cite{1970:Jazwinski:book}:
\begin{align}
    \frac{d\hat x(t)}{dt} & = f\bigl(t,\hat x(t)\bigr),  \label{eq2.1} \\
    \frac{dP(t)}{dt}& = F\bigl(t,\hat x(t)\bigr) P(t)+P(t)F^{\top}\bigl(t,\hat x(t)\bigr)+ GQG^{\top}.  \label{eq2.2}
 \end{align}
The system should be solved for each sampling interval $[t_{k-1},t_k]$ with the initial conditions $\hat x(t_{k-1}) =\hat x_{k-1|k-1}$ and  $P(t_{k-1})=P_{k-1|k-1}$ where the matrix $F\bigl(t,x(t)\bigr):={\partial_{x} f\bigl(t,x(t)\bigr)}$ is the Jacobian matrix. Thus, the continuous-discrete EKF propagation step involves a numerical integration technique for solving~\eqref{eq2.1}, \eqref{eq2.2} on interval $[t_{k-1}, t_{k}]$ in order to determine the one-step ahead predicted estimate  $\hat x_{k|k-1}:=\hat x(t_{k})$ and covariance $P_{k|k-1}:=P(t_{k})$.

The size of system~\eqref{eq2.1}, \eqref{eq2.2} is large and equals to $n+n^2$ where $n$ is a number of unknown states to be estimated. Besides, the equations in~\eqref{eq2.1} may have a complicated underlying dynamics; e.g., they might be highly nonlinear, depending on the given drift function in the state-space model~\eqref{eq1.1}, \eqref{eq1.2}. At the same time, the equations in~\eqref{eq2.2} are linear with respect to elements of the covariance matrix. Thus, it makes sense to handle~\eqref{eq2.1} and~\eqref{eq2.2} separately, paying a special attention to an accurate numerical solution of system~\eqref{eq2.1}. For the EKF strategy it is possible to utilize different numerical integration schemes for solving~\eqref{eq2.1} and~\eqref{eq2.2} because they are not coupled~\cite{2008:Mazzoni}. This simplifies calculations and reduces computational time because a computationally heavy discretization error control is applied for solving~\eqref{eq2.1}, only.

We utilize the advanced variable-stepsize NIRK formulas of order~6 with the built-in automatic combined local-global error control given in~\cite[eqs.(27)~--~(32)]{2017:Kulikov:ANM}. More details of the numerical scheme applied are presented in Appendix. While solving~\eqref{eq2.1} the adaptive  mesh is generated automatically in the sampling interval $[t_{k-1},t_k]$  according to the utilized discretization error control
\[\{t_l\}_{l=0}^{end}:=\left\{t_{l+1}=t_{l}+\tau_l, l=0,1,\ldots,end-1, t_0=t_{k-1},t_{end}=t_{k}\right\} \mbox{ with } \tau_l:=t_{l+1}-t_{l}. \]

The discretization error control techniques keep the  error arisen at the prediction filtering step less than the pre-defined tolerance $\epsilon_g$ given by users. Next, we re-use the mesh~$\{t_l\}_{l=0}^{end}$ for solving~\eqref{eq2.2} with no control technique applied and by a more simple scheme suggested in~\cite{2008:Mazzoni}. More precisely, having computed $\hat x_{l}$ at time $t_{l}$ on the mesh~$\{t_l\}_{l=0}^{end}$, use the same step size $\tau_l$ and the computed stage value ${\hat x}_{l2}^{3}$ from~\cite[eqs.(27)]{2017:Kulikov:ANM}, to calculate the error covariance matrix by the formula
\begin{equation}\label{eq2.21}
P_{l+1}=M_{l+1/2} P_{l} M_{l+1/2}^{\top}+\tau_{l} K_{l+1/2} GQG^{\top} K_{l+1/2}^{\top}
\end{equation}
where $t_{l+1/2}:=t_{l}+\tau_{l}/2$ is the mid-point of the $(l+1)$-st step. We stress that the mid-point $\hat x_{l+1/2}:={\hat x}_{l2}^{3}$ is already computed by the numerical method presented in Appendix, and
\begin{align}
K_{l+1/2} & =\left[I_n - \frac{\tau_{l}}{2}F(t_{l+1/2}, \hat x_{l+1/2})\right]^{-1}, & M_{l+1/2} & = K_{l+1/2}\left[I_n + \frac{\tau_{l}}{2}F(t_{l+1/2}, \hat x_{l+1/2})\right]. \label{eq:Kl}
\end{align}

Further details and summary of the {\it accurate continuous-discrete} EKF time update step performed within the suggested numerical scheme can be found in~\cite[p.~267]{2017:Kulikov:ANM} and in Appendix of this paper. In the next section, we derive a mathematically equivalent implementation method via the covariance SVD factors for improving its numerical robustness in a finite precision arithmetics.

\section{The SVD-based estimation method for the hybrid Extended-Cubature Kalman filter} \label{main:result}

Development of {\it factored-form} (square-root) filtering methods requires a re-derivation of the conventional filtering equations in terms of the matrix square-root factors involved. The methodology yields a reliable estimation approach for processing ill-conditioned state estimation scenarios. In particular, the square-root implementation way ensures the symmetric form and positive (semi-) definiteness of the filters' covariance matrices, reduces sensitivity and improves stability with respect to roundoff.

\begin{Lm}[Time update of the spectral factors] \label{Lemma:1}
Given the measurement updated factors, i.e. the orthogonal $Q_{P_{k-1|k-1}}$ and diagonal $D^{1/2}_{P_{k-1|k-1}}$ of $P_{k-1|k-1}$, set the values $Q_{P_{0}}:=Q_{P_{k-1|k-1}}$ and $D^{1/2}_{P_{0}}:=D^{1/2}_{P_{k-1|k-1}}$ on an integration mesh~$\{t_l\}_{l=0}^{end}:=\left\{t_{l+1}=t_{l}+\tau_l, l=0,1,\ldots,end-1, t_0=t_{k-1},t_{end}=t_{k}\right\}$ in the sampling interval $[t_{k-1},t_k]$. Having known the SVD factors $Q_{P_{l}}$ and $D^{1/2}_{P_{l}}$ of $P_l$ at the current node $t_l$, define the augmented matrix $A_{t_l}$ as
\begin{equation}
A_{t_l} = \left[M_{l+1/2} Q_{P_{l}}D_{P_{l}}^{1/2}, \;\; \sqrt{\tau_{l}}K_{l+1/2}G Q_{Q}D^{1/2}_{Q} \right] \label{svd:TU:pre-array}
\end{equation}
where $Q_{Q}$ and $D^{1/2}_{Q}$ are the SVD factors of the process covariance matrix $Q$ and matrices $M_{l+1/2}$, $K_{l+1/2}$ are computed by formula~\eqref{eq:Kl}. Next, decompose the pre-array $A_{t_l}$ into the SVD-post-arrays as follows:
\begin{equation}
A_{t_l} = W_{t_l} \; [S_{t_l} \quad 0] \; V_{t_l}^{\top},  \label{svd:TU:post-array}
\end{equation}
then the time updated SVD factors of $P_{l+1}$ at the new node $t_{l+1}$ are
\begin{align}
Q_{P_{l+1}} & = W_{t_l}, & D^{1/2}_{P_{l+1}} & = S_{t_l}. \label{proof:TU}
\end{align}
At the last point $t_{end}$ of the mesh introduced in the sampling interval $[t_{k-1},t_k]$, we have $Q_{P_{k|k-1}}:=Q_{P_{end}}$ and $D^{1/2}_{P_{k|k-1}}:=D^{1/2}_{P_{end}}$.
\end{Lm}
\begin{proof}
Let us consider equation~\eqref{eq2.21} for calculating $P_{l+1}$. It can be written by using the spectral factors as follows:
\begin{align}
Q_{P_{l+1}}D_{P_{l+1}}Q_{P_{l+1}}^{\top} & = M_{l+1/2} Q_{P_{l}}D_{P_{l}}Q_{P_{l}}^{\top} M_{l+1/2}^{\top} +\tau_{l} K_{l+1/2} GQ_{Q}D_{Q}Q_{Q}^{\top}G^{\top} K_{l+1/2}^{\top}.
\end{align}
Since the equation has a symmetric form, we may factorize
\begin{align*}
Q_{P_{l+1}}D_{P_{l+1}}Q_{P_{l+1}}^{\top} & = \left[M_{l+1/2} Q_{P_{l}}D_{P_{l}}^{1/2}, \;\; \sqrt{\tau_{l}}K_{l+1/2}G Q_{Q}D^{1/2}_{Q} \right] \left[M_{l+1/2} Q_{P_{l}}D_{P_{l}}^{1/2}, \;\; \sqrt{\tau_{l}}K_{l+1/2}G Q_{Q}D^{1/2}_{Q}\right]^{\top}
\end{align*}
where the SVD factors $Q_{P_{l}}$ and $D_{P_{l}}$ of the matrix $P_{l}$ are assumed to be known from the previous recursion step. Having denoted the multiplier by
\begin{equation} \label{eq:A1}
A_{t_l} = \left[M_{l+1/2} Q_{P_{l}}D_{P_{l}}^{1/2}, \;\; \sqrt{\tau_{l}}K_{l+1/2}G Q_{Q}D^{1/2}_{Q}\right],
\end{equation}
we have $Q_{P_{l+1}}D_{P_{l+1}}Q_{P_{l+1}}^{\top} = A_{t_l}A_{t_l}^{\top}$. Having replaced $A_{t_l}$ with its SVD factors in~\eqref{svd:TU:post-array}, we get  $A_{t_l}A_{t_l}^{\top} = W_{t_l} S_{t_l}^2 W_{t_l}^{\top}$ and the result follows, i.e. $W_{t_l}:=Q_{P_{l+1}}$ and $S_{t_l}^2:=D_{P_{l+1}}$.
\end{proof}

\begin{table}[ht!]
\caption{A summary of calculations of the SVD-based NIRK-type EKF-CKF algorithm} \label{tab:1}                                                                                                                                                                                                                                                                                                                                                                                                                                                                                                                                                                                                                                                                \begin{center}
\begin{tabular}{ll}
\hline
\texttt{Initial conditions}: & Set the initial $\hat x_{0|0} = \bar x_0$ and SVD factors $Q_{P_{0|0}} = Q_{\Pi_0}$, $D^{1/2}_{P_{0|0}} = D^{1/2}_{\Pi_0}$; \\
\hhline{~|-}
\texttt{Time Update} ($[t_{k-1},t_k]$):  &  \textit{at each node $t_l$ of the generated mesh $\{t_l\}_{l=0}^{end}$ perform} \\
- State Estimate & Given tolerance $\epsilon_g$, integrate~\eqref{eq2.1} with the error control in~\cite[eqs.(27)~--~(32)]{2017:Kulikov:ANM}; \\
- SVD processor & $A_{t_l} \leftarrow W_{t_l} \; [S_{t_l} \quad 0] \; V_{t_l}^{\top}$ where $A_{t_l} = \left[M_{l+1/2} Q_{P_{l}}D_{P_{l}}^{1/2},\quad  \sqrt{\tau_{l}}K_{l+1/2}G Q_{Q}D^{1/2}_{Q} \right]$;\\
& Read-off: $Q_{P_{l+1}}= W_{t_l}$, $D^{1/2}_{P_{l+1}} = S_{t_l}$. Set $Q_{P_{k|k-1}}=Q_{P_{end}}$, $D^{1/2}_{P_{k|k-1}}=D^{1/2}_{P_{end}}$; \\
\hhline{~|-}
\texttt{Measurement Update}: & Apply CKF formulas~\eqref{start:cub}~--~\eqref{end:cub}; \\
- SVD processor & $B_{t_l} \leftarrow W_{t_l} \; [S_{t_l} \quad 0] \; V_{t_l}^{\top}$  where the pre-array is $B_{t_l} = \left[{\mathbb Z}_{k|k-1}, \;\; Q_{R_k}D_{R_k}^{1/2}\right]$;  \\
& Read-off $Q_{R_{e,k}}= W_{t_l}$, $D^{1/2}_{R_{e,k}} = S_{t_l}$ and find the cross-covariance in~\eqref{cross-cov-ckf};\\
- State Estimate &  $\hat x_{k|k}=\hat x_{k|k-1}+{\mathbb K}_k(z_k-\hat z_{k|k-1})$ where ${\mathbb K}_{k}=P_{xz,k}Q_{R_{e,k}}D_{R_{e,k}}^{-1}Q_{R_{e,k}}^{\top}$; \\
- SVD processor & $C_{t_l} \leftarrow W_{t_l} \; [S_{t_l} \quad 0] \; V_{t_l}^{\top}$ where $C_{t_l} = \left[{\mathbb X}_{k|k-1}-{\mathbb K}_{k}{\mathbb Z}_{k|k-1}, \quad {\mathbb K}_{k}Q_{R_k}D_{R_k}^{1/2}\right]$;  \\
& Read-off the updated SVD factors: $Q_{P_{k|k}}= W_{t_l}$, $D^{1/2}_{P_{k|k}} = S_{t_l}$. \\
\hline
\end{tabular}
\end{center}
\end{table}

\begin{Lm}[Measurement update of the spectral factors] \label{Lemma:2}
Given the predicted factors: the orthogonal $Q_{P_{k|k-1}}$ and diagonal $D^{1/2}_{P_{k|k-1}}$ of $P_{k|k-1}$, generate the cubature nodes by~\eqref{start:cub}, propagate through nonlinear measurement function in~\eqref{propagateH} and find the centered matrix ${\mathbb Z}_{k|k-1} \in {\mathbb R}^{m\times 2n}$ by formula~\eqref{matZ}. Next, define the augmented matrix $B_{t_l}$ as follows:
\begin{equation}
B_{t_l} = \left[{\mathbb Z}_{k|k-1}, \;\; Q_{R_k}D_{R_k}^{1/2}\right] \label{svd:MU1:pre-array}
\end{equation}
where $Q_{R_k}$ and $D^{1/2}_{R_k}$ are the SVD factors of the measurement noise covariance $R_k$. Next, use SVD to decompose the pre-array $B_{t_l}$ into the post-arrays as follows:
\begin{equation}
B_{t_l} = W_{t_l} \; [S_{t_l} \quad 0] \; V_{t_l}^{\top},  \label{svd:MU1:post-array}
\end{equation}
then the SVD factors of the residual covariance $R_{e,k}$ in~\eqref{ckf:rek} are
\begin{align}
Q_{R_{e,k}} & = W_{t_l}, & D^{1/2}_{R_{e,k}} & = S_{t_l}. \label{proof:MU1}
\end{align}

Having computed $Q_{R_{e,k}}$ and $D^{1/2}_{R_{e,k}}$, we can calculate the cross-covariance in~\eqref{cross-cov-ckf} and the cubature gain in the stable square-root form as follows:
\begin{equation}
{\mathbb K}_{k}=P_{xz,k}Q_{R_{e,k}}D_{R_{e,k}}^{-1}Q_{R_{e,k}}^{\top} \label{cub:gain:svd:new}
\end{equation}
and next define the augmented pre-array $C_{t_l}$ as
\begin{equation}
C_{t_l} = \left[{\mathbb X}_{k|k-1}-{\mathbb K}_{k}{\mathbb Z}_{k|k-1}, \;\; {\mathbb K}_{k}Q_{R_k}D_{R_k}^{1/2}\right] \label{svd:MU2:pre-array}
\end{equation}
where the centered matrices ${\mathbb Z}_{k|k-1}$ and ${\mathbb X}_{k|k-1}$ are calculated by the CKF rule and formula~\eqref{matZ}.
Again, use SVD to decompose the pre-array $C_{t_l}$ into the following post-arrays:
\begin{equation}
C_{t_l} = W_{t_l} \; [S_{t_l} \quad 0] \; V_{t_l}^{\top},  \label{svd:MU2:post-array}
\end{equation}
then the SVD factors of the filtered error covariance matrix $P_{k|k}$ are
\begin{align}
Q_{P_{k|k}} & = W_{t_l}, & D^{1/2}_{P_{k|k}} & = S_{t_l}. \label{proof:MU2}
\end{align}
 \end{Lm}
\begin{proof}
The derivation is similar to the proof of Lemma~\ref{Lemma:1}  and follows from equation~\eqref{ckf:rek} for calculating $R_{e,k}$ and from the symmetric form  equation for calculating $P_{k|k}$ derived in~\cite{2009:Haykin}:
\begin{align}
 P_{k|k} & = \left({\mathbb X}_{k|k-1}-{\mathbb K}_{k}{\mathbb Z}_{k|k-1}\right) \left({\mathbb X}_{k|k-1}-{\mathbb K}_{k}{\mathbb Z}_{k|k-1}\right)^{\top} + {\mathbb K}_{k}R_k {\mathbb K}_{k}^{\top}. \label{eq:pcov:cubature}
\end{align}

The detailed proof of Lemma~\ref{Lemma:2} can be also found in~\cite{2020:Automatica:Kulikova}.
\end{proof}

Table~\ref{tab:1} summarizes the {\it continuous-discrete} mixed-type EKF-CKF method proposed.

\section{Numerical experiments} \label{numerical:experiments}

The purpose of this section is to assess a performance of the newly-suggested spectral mixed-type EKF-CKF method in comparison with the existing CKF algorithms: (i) the original continuous-discrete CKF based on the It\^{o}-Taylor expansion (IT-1.5 CKF) suggested in~\cite{2010:Haykin}, and (ii) its SVD-based variant (IT-1.5 CKF SVD) recently derived in~\cite{2020:Automatica:Kulikova}. For a fair comparative study, we examine three test problems taken from various application fields.

\begin{exmp}[Target Tracking Test] \label{ex:1} When performing a coordinated turn in the horizontal plane, the aircraft's dynamics obeys equation~\eqref{eq1.1} with the following drift function and diffusion matrix
\[
f(\cdot)=\left[\dot{\epsilon}, \; -\omega \dot{\eta}, \;\dot{\eta},\; \omega \dot{\epsilon}, \;\dot{\zeta}, \; 0, \; 0\right] \mbox{ and }
G={\rm diag}\left\{0,\; \sigma_1, \; 0, \; \sigma_1, \; 0, \; \sigma_1, \; \sigma_2\right\}
\] where $\sigma_1=\sqrt{0.2}$, $\sigma_2=0.007$ and $\beta(t)$ is the {\it standard} Brownian motion, i.e. $Q=I_7$.
The state vector consists of seven entries, i.e. $x(t)= [\epsilon, \; \dot{\epsilon}, \; \eta, \; \dot{\eta}, \; \zeta, \; \dot{\zeta}, \; \omega]^{\top}$, where $\epsilon$, $\eta$, $\zeta$ and $\dot{\epsilon}$, $\dot{\eta}$, $\dot{\zeta}$ stand for positions and corresponding velocities in the Cartesian coordinates at time $t$, and $\omega(t)$ is the (nearly) constant turn rate. The initial conditions are $\bar x_0=[1000\,\mbox{\rm m}, 0\,\mbox{\rm m/s}, 2650\,\mbox{\rm m},150\,\mbox{\rm m/s}, 200\,\mbox{\rm m}, 0\,\mbox{\rm m/s},\omega^\circ/\mbox{\rm s}]^{\top}$ and $\Pi_0=\mbox{\rm diag}(0.01\,I_7)$. We fix the turn rate to $\omega=3^\circ/\mbox{\rm s}$.

{\bf Case 1: The original problem}. The measurement model is taken from~\cite{2010:Haykin}, but accommodated to MATLAB as follows:
\[
\begin{bmatrix}
r_k \\
\theta_k \\
\phi_k
  \end{bmatrix}
 =
  \begin{bmatrix}
  \sqrt{\epsilon^2_k+\eta^2_k+\zeta^2_k} \\
  {\rm atan2}\left({\eta_k},{\epsilon_k}\right) \\
  {\rm atan}\left(\frac{\zeta_k}{{\sqrt{\epsilon^2_k+\eta^2_k}}}\right)
  \end{bmatrix}
  + v_k,
  \begin{array}{l}
    v_k \sim {\cal N}(0,R); \\
    R  ={\rm diag}\left\{\sigma_r^2,\sigma_\theta^2,\sigma_\phi^2\right\}
\end{array}
\]
where the implementation of the MATLAB command {\tt atan2} in target tracking is explained in~\cite{brehard2007hierarchical} in more details.
The observations $z_k = [r_k, \theta_k, \phi_k]^{\top}$, $k=1,\ldots, K$ come at some constant sampling intervals $\Delta = [t_{k-1}, t_k]$. The radar is located at the origin and is equipped to measure the range $r_k$, azimuth angle~$\theta$ and the elevation angle~$\phi$. The measurement noise covariance matrix is constant over time and $\sigma_r=50$~m, $\sigma_\theta=0.1^\circ$, $\sigma_\phi=0.1^\circ$.

{\bf Case 2: The ill-conditioned tests}. Following the ill-conditioned test problem proposed in~\cite[Example~7.2]{2015:Grewal:book} and discussed at the first time in~\cite[Examples 7.1 and 7.2]{1969:Dyer}, we design the following measurement scheme for provoking the filters' numerical instability due to roundoff:
\begin{align*}
z_k & =
\begin{bmatrix}
1 & 1 & 1 & 1 & 1 &  1 &  1\\
1 & 1 & 1 & 1 & 1 &  1 &  1 +\delta
\end{bmatrix}
x_k +
\begin{bmatrix}
v_k^1 \\
v_k^2
\end{bmatrix}, \;
  \begin{array}{l}
    v_k \sim {\cal N}(0,R); \\
     R=\delta^{2}I_2
\end{array}
\end{align*}
where parameter $\delta$ is used for simulating roundoff effect. This increasingly ill-conditioned target tracking scenario assumes that $\delta\to 0$, i.e. $\delta=10^{-1},10^{-2},\ldots,10^{-13}$.
\end{exmp}

First, we consider the original problem and repeat the experiments in~\cite{2010:Haykin,2020:Automatica:Kulikova} for estimating the unknown state vector $\hat x_{k|k}$ on the interval $[0s, 150s]$ with varying sampling periods $\Delta =1, \ldots, 12$ (s). Following~\cite{2010:Haykin}, the square estimation errors are calculated at all sampling points by taking the norm of the difference between the  ``true'' solution $x^{true}_k$ and the estimated one $\hat x_{k|k}$ for all $k=1, \ldots, K$. Next, $M=100$ Monte Carlo simulations are performed for calculating the accumulated root mean square error (ARMSE) by averaging over $100$ trials. Following~\cite{2010:Haykin}, we additionally compute the ARMSE in position, i.e. $\mbox{\rm ARMSE}_p$, as follows:
\begin{align}
\mbox{\rm ARMSE}_p & =\Bigl[\frac{1}{M
K}\sum_{M=1}^{100}\sum_{k=1}^K\bigl(\epsilon^{\rm true}_k-\hat
\epsilon_{k|k}\bigr)^2  +\bigl(\eta^{\rm true}_k-\hat
\eta_{k|k}\bigr)^2\!+\bigl(\zeta^{\rm true}_k-\hat
\zeta_{k|k}\bigr)^2\Bigr]^{1/2}. \label{eq:acc}
\end{align}
and define the filters failure when $\mbox{\rm ARMSE}_p> 500$ (m).

 All estimators under examination are tested at the same conditions, i.e. with the same simulated ``true'' state trajectory, the same measurement data and the same initial conditions. The fixed stepsize IT-1.5 CKF methods are implemented with $m=64$ subdivisions. The NIRK-based EKF-CKF does not require the number of subdivisions to be given but it demands the tolerance value. We implement the variable stepsize algorithms with $\epsilon_g = 10^{-4}$. The results of this set of numerical experiments are summarized in Table~\ref{tab:acc} where the estimation accuracies and the average CPU time (in sec.) are collected. We also calculate the estimation accuracy and CPU time benefits (in \%) of using the suggested in this paper mixed-type EKF-CKF instead of the original CKF method, i.e. we compute $\texttt{CKF}_{accuracy|CPU}/\texttt{EKF-CKF}_{accuracy|CPU}-1$  in per cents.

\begin{table}[t!]
\caption{The computed $\mbox{\rm ARMSE}_p$ (m), average CPU time (s) and the benefits (\%) of the mixed EKF-CKF ($\epsilon_g = 10^{-4}$) over the IT-1.5 CKF ($m=64$) obtained for {\bf Case 1} study in Example~1.} \label{tab:acc}
{\scriptsize
\begin{tabular}{r||r|r|r|r|r||r|r|r|r|r}
\hline
& \multicolumn{5}{c||}{\bf The computed accuracy $\mbox{\rm ARMSE}_p$ (m)} & \multicolumn{5}{c}{\bf The average CPU time (s)}\\
\hline
$\Delta$(s) & \multicolumn{2}{c|}{\bf IT-1.5 CKF} & \multicolumn{2}{c|}{\bf mixed EKF-CKF} & {\bf Benefit} & \multicolumn{2}{c|}{\bf IT-1.5 CKF} & \multicolumn{2}{c|}{\bf mixed EKF-CKF} & {\bf Benefit} \\
\cline{2-11}
& original~\cite{2010:Haykin} & SVD~\cite{2020:Automatica:Kulikova} & original~\cite{2017:Kulikov:ANM} & {\bf new} SVD & (\%) & original~\cite{2010:Haykin} & SVD~\cite{2020:Automatica:Kulikova} & original~\cite{2017:Kulikov:ANM} & {\bf new} SVD & (\%) \\
\hline
    2 &  82.9 &  82.9 &  93.4 &  93.4 & -11\% &    0.52&    0.59&    0.72&    0.84&  -28\% \\
    4 &  95.0 &  95.0 & 113.0 & 113.0 & -15\% &    0.25&    0.29&    0.53&    0.62&  -52\%\\
    6 & 110.4 & 110.4 & 127.3 & 127.3 & -13\% &    0.17&   0.19&    0.48&    0.56&  -64\%\\
    8 & 121.8 & 121.8 & 142.8 & 142.8 & -14\% &    0.12&    0.14&    0.43&    0.49&  -70\%\\
   10 & 160.0 & 160.0  & 144.4 & 144.4 &  10\% &    0.10&    0.11&    0.42&    0.48&  -75\% \\
   \hline
   12 & $>$500 & $>$500  & 159.4 & 159.4 & ---  &  --- &  ---  &  0.40 & 0.45 & --- \\
     & {\bf fails} & {\bf fails} & & & & {\bf fails}& {\bf fails}  &   &  & --- \\
   \hline
\end{tabular}
}
\end{table}


Having analyzed the results obtained for {\bf Case 1} study in Example~1, we make a few conclusions:
 \begin{itemize}
 \item In average, the IT-1.5 CKF filtering approach on $\approx 12$\%  outperforms the hybrid EKF-CKF estimator for the estimation accuracy on short sampling periods. It is inline with nonlinear filtering theory because the moment approximation error arisen within the CKF methodology is less than within the mixed-type EKF-CKF strategy. However, this holds true only for small sampling intervals. As can be seen, for $\Delta =10$(s) the NIRK-based hybrid filters outperform the IT-1.5 CKF estimators for accuracy on about $10$\% and for $\Delta \ge 12$(s) the original and SVD-based IT-1.5 CKFs ultimately fail. It happens due to the accumulated discretization error and the lack of controlling techniques. The accumulated errors destroy any CKF-type filters based on the SDEs numerical integration schemes.
 \item The main problem of all It\^{o}-Taylor CKF implementations is that the sufficient number of subdivisions heavily depends on the problem at hand, and this optimal value should be defined in advance. Any IT-1.5 CKF implementation is accurate when the pre-fixed number of subdivisions is enough to ensure a small discretization error at the propagation step and it rapidly fails when this is not a case. In contrast, the filters based on numerical solution of the related moment differential equations allow for discretization error control and, hence, they suggest an accurate estimation way for any sampling interval length including the scenario with irregular sampling periods due to missing measurements.
 \item The NIRK-based EKF-CKF and its SVD-based counterpart work accurately for any sampling interval length. Indeed, the built-in advanced numerical integration scheme with the adaptive error control strategy ensures (in automatic mode) that the occurred discretization error is insignificant.          The adaptive nature of the NIRK-based EKF-CKF estimators make them flexible and convenient for using in practice. No preliminary tuning is required by users, except the tolerance value to be given. However, the error control technique involved in the NIRK-based EKF-CKF filters make them $\approx 0.5$ times slower than the methods without estimation accuracy control. This is the price to be paid for an accurate time propagation step. Evidently, this value grows as the sampling interval becomes longer because the accuracy requirements at the time update steps make the estimation problem hard.
\item The fixed-stepsize IT-1.5 CKF schemes are fast to execute, but the estimation quality might be very poor (e.g., for irregular sampling intervals). As can be seen, for $\Delta =10$ the NIRK-based EKF-CKF is on $\approx 10\%$ more accurate than the IT-1.5 CKF, but it is slower with the factor $0.75$. However, for the next $\Delta =12$ the IT-1.5 CKF method fails although the NIRK-based EKF-CKF works accurately and it takes $\approx 0.5$ (s), in average. We may conclude that the suggested NIRK-based filtering provides a good balance between the improved estimation accuracy including some useful features (e.g., an ability to manage irregular and long sampling intervals) and computational demands.
\end{itemize}

Let us next explore Case 2 study in Example~1. Following the results in Table~\ref{tab:acc}, the CPU time of the SVD-based implementations is a bit higher than in the related standard algorithms without SVD factorization. Meanwhile the estimation accuracies of the SVD-based and standards algorithms are the same for a well-conditioned tests in Case 1 study of Example~1. This substantiates the algebraic equivalence between the standard and SVD-based implementation frameworks proved in Section~3. To investigate the difference in the numerical robustness (with respect to roundoff errors) of the filtering methods under examination, we explore the  ill-conditioned estimation scenarios. In engineering literature, the first ill-conditioned measurement schemes have been discussed in details in~\cite[Examples 7.1 and 7.2]{1969:Dyer}. Nowadays, these ill-conditioned tests are widely used for a comparative study of  any KF-like estimator; see the third reason of the KF divergence due to singularity arisen in the covariance matrix $R_{e,k}$ caused by roundoff as discussed in~\cite[p.~288]{2015:Grewal:book}. More precisely, when the ill-conditioning parameter $\delta$ tends to a machine precision limit,
one observes a degradation of any KF-like method due to roundoff errors and singularities appeared in the matrix $R_{e,k}$ to be inverted for computing the gain matrix $\mathbb K_k$. Meanwhile any spectral SVD-based filter implies the inverse of the diagonal SVD factor of the matrix $R_{e,k}$, i.e. $D_{R_{e,k}}$, only. The readers are refereed to equation~\eqref{cub:gain:svd:new} for more details. Thus, we are going to observe a difference between two computational strategies: the conventional filtering and the SVD-based spectral implementations. For that, we repeat the numerical experiments in Case~1 of Example~1 but for the following ill-conditioning parameter $\delta$ values: (i) $\delta = 10^{-1}$ represents a well-conditioned scenario, (ii) $\delta = 10^{-5}$ corresponds to a moderate problem ill-conditioning, and (iii) $\delta = 10^{-12}$ yields a strong ill-conditioned scenario.

\begin{figure}
\begin{tabular}{cc}
\includegraphics[width=0.5\textwidth]{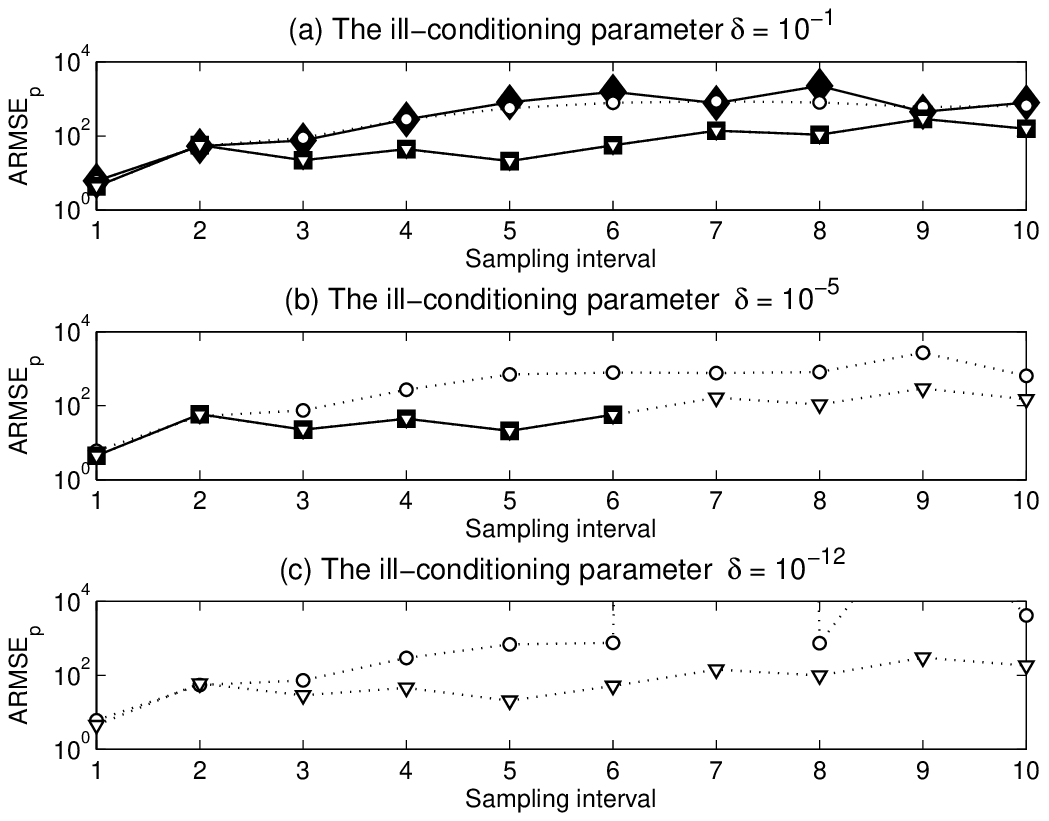} & \includegraphics[width=0.5\textwidth]{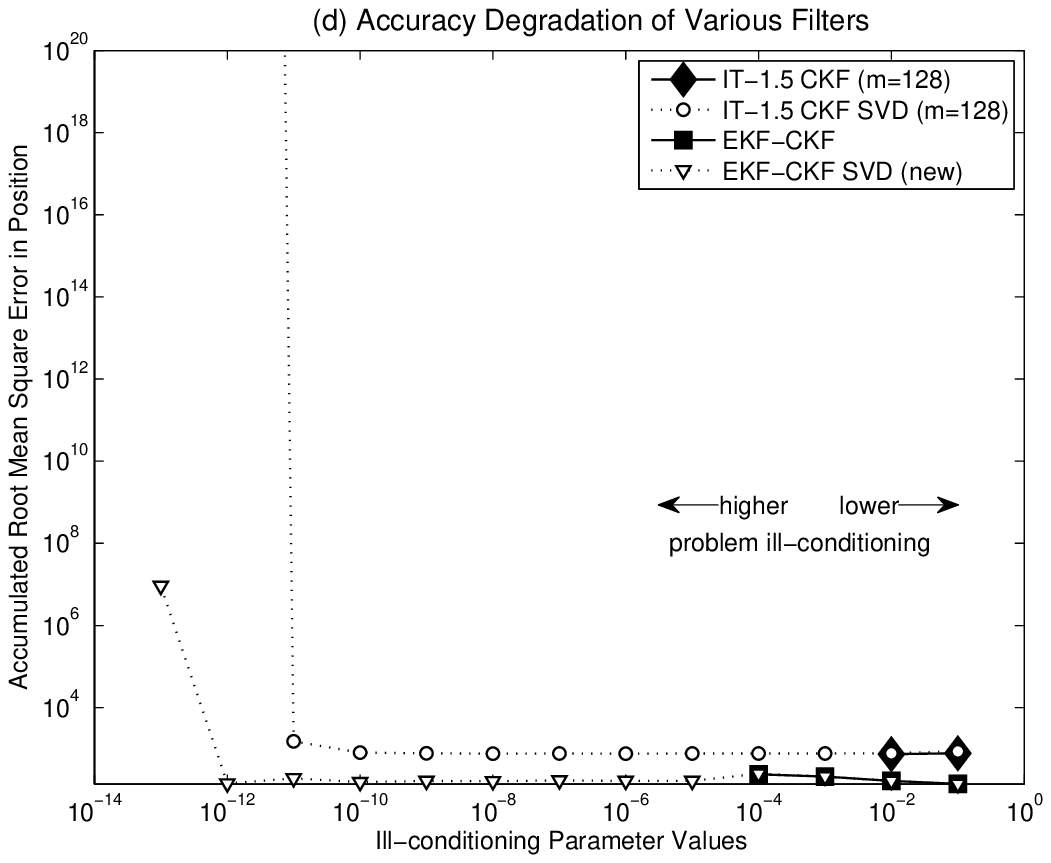}
\end{tabular}
\caption{The degradation of the filters' accuracy while ill-conditioned scenario provided by Case~2 in Example~\ref{ex:1}. The left figure: the IT-1.5 CKF with $m=128$ (solid line with $\blacklozenge$), the SVD-based IT-1.5 CKF with $m=128$ (dotted line with $\bigcirc$), the mixed EKF-CKF with $\epsilon_g = 10^{-4}$ (solid line with $\blacksquare$) and new SVD EKF-CKF $\epsilon_g = 10^{-4}$ (dotted line with $\bigtriangledown$). The right figure: the filters degradation when $\Delta =7$(sec) for various ill-conditioning parameter values $\delta$, which tend to machine precision.} \label{fig:2}
\end{figure}

Having analyzed Fig.~\ref{fig:2}(a), we conclude that all filtering algorithms under examination work accurately in the well-conditioned scenario, i.e. when $\delta$ is large. It is interesting to note that in contrast to Case~1 study, we observe a better performance of the NIRK-based EKF-CKF filters then the IT-1.5 CKF methods (with $m=128$ subdivisions). Then, Fig.~\ref{fig:2}(b) illustrates a degradation of all conventional implementations while the problem ill-conditioning grows. As can be seen, the original IT-1.5 CKF with $m=128$ is not able to solve the problem even for $\Delta =1(s)$ and, hence, it is not plotted on Fig.~\ref{fig:2}(b). Meanwhile the conventional EKF-CKF accurately solves the estimation problem for short sampling intervals, only. In contrast, their SVD-based counterparts maintain a high estimation quality for any sampling interval length, i.e. they accurately treat the moderate ill-conditioned scenario in a robust way. Finally, from Fig.~\ref{fig:2}(c) we observe that only the new SVD-based EKF-CKF filter is able to solve the strong ill-conditioned estimation problem. Any conventional implementation fails at $\delta = 10^{-12}$ even for the short sampling interval $\Delta=1(s)$ and, hence, they are not plotted on Fig.~\ref{fig:2}(c). More importantly, the SVD-based IT-1.5 CKF implementation previously suggested in~\cite{2020:Automatica:Kulikova} is also not stable and deviates from the 'true' state for some $\Delta$ values significantly.  It seems that the arisen discretization and roundoff errors are accumulated in the fixed stepsize filters based on the It\^{o}-Taylor expansion and this amplifies their effect yielding the filters failure.

In our last set of numerical experiments, we fix the sampling interval to $\Delta = 7(s)$ and examine the filters numerical stability in a finite precision arithmetics. Fig.~\ref{fig:2}(d) illustrates the resulted errors $\mbox{\rm ARMSE}_p$ computed for various values of the ill-conditioning parameter $\delta$ when it tends to a machine precision limit. Having analyzed the obtained outcomes, we conclude that the original IT-1.5 CKF algorithm (with $m=128$ subdivisions) is the fastest method to diverge due to roundoff. It is able to solve the ill-conditioned test problems when $\delta$ is large, i.e. $\delta \ge 10^{-2}$. Next, the conventional NIRK EKF-CKF filter is a bit more stable but it fails when $\delta < 10^{-4}$. As can be seen, the SVD-based implementations are the most robust methods with respect to roundoff. Both the SVD-based IT-1.5 CKF and the novel SVD- NIRK-based EKF-CKF designed in this paper work accurately and manage the state estimation problem till $\delta = 10^{-11}$. Again, the NIRK-based EKF-CKF filter provides a better estimation quality than the CKF estimator. Finally, the SVD-based IT-1.5 CKF proposed in~\cite{2020:Automatica:Kulikova} fails at $\delta = 10^{-12}$ meanwhile the novel SVD EKF-CKF manages this ill-conditioned estimation problem accurately.

\begin{exmp}[Gas-phase reversible reaction in CSTR] \label{ex:2}
Following~\cite{KuKu19IJRNC}, the gas-phase reversible reaction with three species denoted as $A$, $B$ and $C$ are given as follows:
\begin{equation}\label{eq2.1}
A\quad {{k_1\atop\rightleftharpoons}\atop {\scriptstyle k_2}}\quad B+C,\quad
2B\quad {{k_3\atop\rightleftharpoons}\atop {\scriptstyle k_4}}\quad B+C,
\end{equation}
where the fixed coefficients $k_1=0.5$, $k_2=0.05$, $k_3=0.2$ and $k_4=0.01$. The stoichiometric matrix $\nu$ of reaction (\ref{eq2.1}) and the reaction rates $r$ are chosen to be
\begin{equation}\label{eq2.2}
\nu=\left[
\begin{array}{ccc}
-1 & 1 & 1\\
0 & -2 & 1
\end{array}
\right], \quad
r=\left[
\begin{array}{c}
k_1 c_A  -k_2 c_B c_C\\
k_3 c_B^2 -k_4 c_C
\end{array}
\right],
\end{equation}
in which $c_A$ stands for the concentration of the species $A$ in moles per liter, and so on. Thus, the state of this chemical system is defined by the vector $x(t)=\left[c_A\quad c_B\quad c_C\right]^\top$.

We explore reaction~\eqref{eq2.1}, \eqref{eq2.2} in a continuously stirred tank reactor (CSTR). The well-mixed, isothermal CSTR is simulated by the following SDE:
\begin{equation}\label{eq2.5:new}
dx(t)=\left[\frac{Q_f}{V_R} c_f - \frac{Q_0}{V_R} x(t) + \nu^\top r\right] dt+G dw(t), \quad t>0,
\end{equation}
where $c_f=x_0$, the coefficients $Q_f=Q_0=1$, $V_R=100$, and the matrix $\nu$ and the vector $r$ are given in (\ref{eq2.2}). The continuous-time process noise $dw(t)$ is the zero-mean white Gaussian process with the diagonal covariance matrix $Q\,dt=\mbox{\rm diag}\{10^{-6}/\delta,10^{-6}/\delta,10^{-6}/\delta\}\,dt$, and the constant diffusion matrix is $G=I_3$. The initial values are $\bar x_0 = [0.5, 0.05, 0]^{\top}$ and $\Pi_0 = I_3$.

{\bf Case 1: The original problem}. The measurement equation is given in the following form:
\begin{equation}\label{eq4.5:new}
z_k=\left[RT\quad RT\quad RT\right] x_k + v_k, \quad
  \begin{array}{l}
    v_k \sim {\cal N}(0,R_z); \\
    R_z  = 0.25^2
\end{array}
\end{equation}
where $R$ is the ideal gas constant and $T$ is the reactor temperature in Kelvin, i.e. $RT=32.84$.

{\bf Case 2: The ill-conditioned tests}.  We design the following measurement scheme for provoking the filters' numerical instability due to roundoff:
\begin{align*}
z_k & =
RT \begin{bmatrix}
1 & 1 & 1 \\
1 & 1 & 1 +\delta
\end{bmatrix}
x_k +
\begin{bmatrix}
v_k^1 \\
v_k^2
\end{bmatrix}, \;
  \begin{array}{l}
    v_k \sim {\cal N}(0,R_z); \\
     R_z=\delta^{2}I_2
\end{array}
\end{align*}
where parameter $\delta$ is used for simulating roundoff effect. This increasingly ill-conditioned target tracking scenario assumes that $\delta\to 0$, i.e. $\delta=10^{-1},10^{-2},\ldots,10^{-13}$.\end{exmp}

We first explore Case~1 scenario of Example~2 and perform the following numerical tests. The SDE in~\eqref{eq2.5:new} is simulated with a small stepsize $\delta_t = 10^{-3}$ on interval $[0, 30]$(s) to generate the true state vector. Next, equation~\eqref{eq4.5:new} is utilized for creating the history of simulated measurements with various sampling rates $\Delta =0.5, 1,\ldots, 4.5, 5$(s). For each fixed $\Delta$(s) value, the filtering problem is solved to get the estimated hidden state. We compute the accumulated root mean square error (ARMSE) by averaging over $100$ Monte Carlo runs and three entries of the state vector $x(t) = \left[c_A\quad c_B\quad c_C\right]^\top$ as follows:
 \begin{align}
\mbox{\rm ARMSE} & =\Bigl[\frac{1}{M
K}\sum_{M=1}^{100}\sum_{k=1}^K\sum_{j=1}^{n}\bigl(x^{\rm true}_{k,j}-\hat
x_{k|k,j}\bigr)^2\Bigr]^{1/2} \label{eq:acc:1}
\end{align}
where the subindex $j$, $j =1,\ldots,n$, refers to the $j$th entry of the $n$-dimensional state vector.

\begin{figure}
\begin{tabular}{cc}
\includegraphics[width=0.5\textwidth]{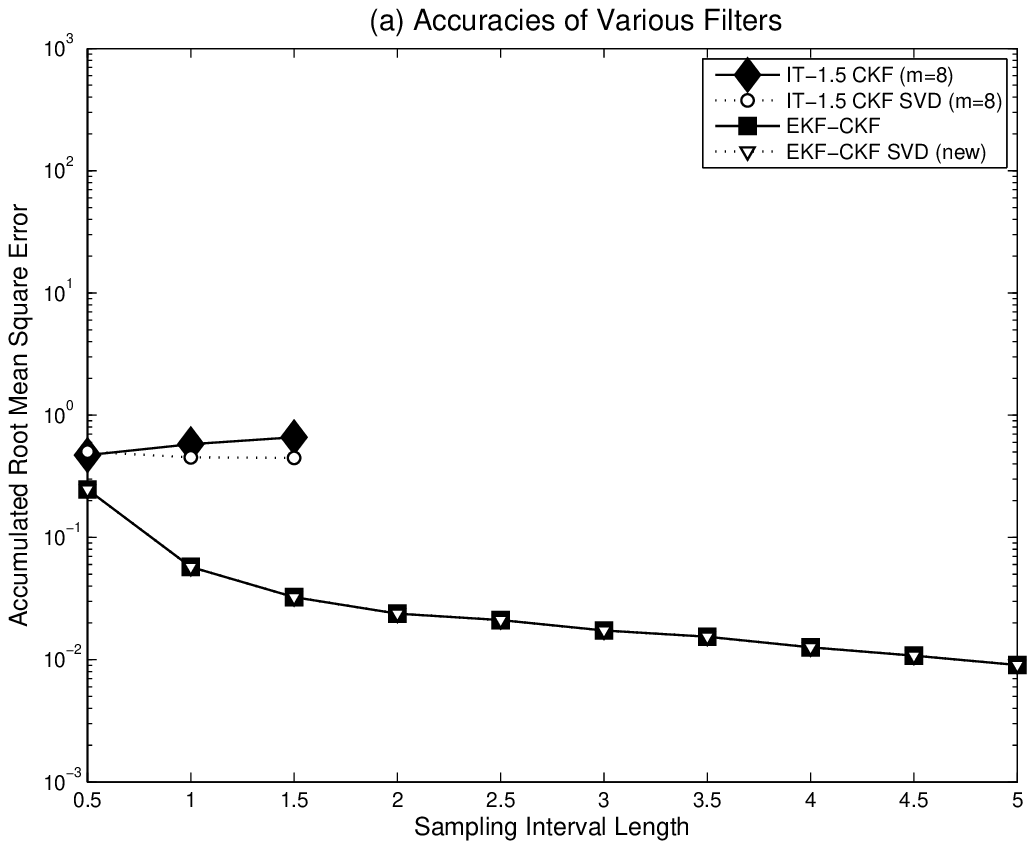} & \includegraphics[width=0.5\textwidth]{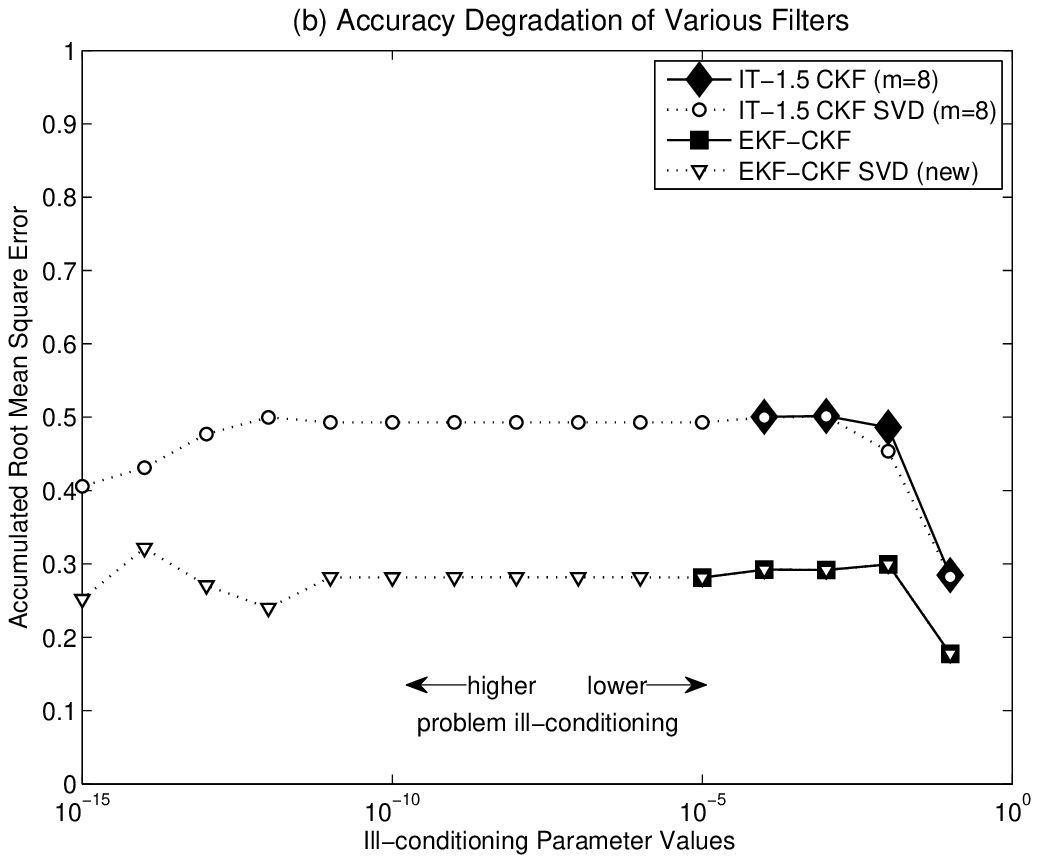}
\end{tabular}
\caption{The filters' performances in Example~\ref{ex:2}: the left graph illustrates Case~1 results, meanwhile the right plot  demonstrates the results of ill-conditioned tests under Case~2 scenario.} \label{fig:2:new}
\end{figure}

Fig.~\ref{fig:2:new}(a) illustrates the ARMSE values computed for each sampling period $\Delta$ mentioned above by using the filtering methods under examination.
In contrast to Example~1, the hybrid NIRK-based EKF-CKF estimators outperform the IT-1.5 CKF filters
for the estimation accuracy on this application example. Besides, the fixed stepsize  IT-1.5 CKF filters quickly fail to solve the filtering problem due to unfeasible Cholesky and SVD factorizations. The discretization error arisen for the sampling rates $\Delta> 1.5$(s) destroys the error covariance matrix to be decomposed and this yields a divergence of the filtering methods based on the  IT-1.5 scheme. Meanwhile, the NIRK-based EKF-CKF filters work accurately for any sampling interval length under examination due to the built-in discretization error control.

We next consider the results of ill-conditioned tests performed for Case~2 scenario in Example~2. They are illustrated by Fig.~\ref{fig:2:new}(b). As can be seen, the NIRK-based hybrid estimators are again more accurate than the IT-1.5 CKF filters. Besides, the conventional NIRK-based CKF-EKF and IT-1.5 CKF degrade faster than their SVD-based counterparts as the ill-conditioning parameter $\delta$ tends to machine precision. Indeed, the conventional IT-1.5 CKF fails for ill-conditioned state estimation problems with $\delta < 10^{-4}$. The conventional NIRK EKF-CKF filter is a bit more stable
and fails when $\delta < 10^{-5}$. Their square-root SVD-based variants are robust to roundoff errors, i.e. they are able to solve the ill-conditioned estimation problems in accurate and robust way. Following Fig.~\ref{fig:2:new}(b), we conclude that the novel SVD- NIRK-based EKF-CKF is numerically stable and the most accurate estimator among all filtering methods under examination.

Our last numerical example is focused on exploring capacities of the above filters for estimating stiff stochastic models. For that, we examine the stochastic Van der Pol oscillator, which is considered to be a classical benchmark in nonlinear filtering theory
by many authors~\cite{Frog2012,Ma08}. This test example can expose both nonstiff and stiff behaviors, depending on the value of its stiffness parameter $\lambda$. So difficulties of state estimation in stiff stochastic systems are observed by comparing performances of the filtering methods while $\lambda$ increases the period of oscillations. Our test example is rescaled as explained in~\cite[p.~5]{HaWa96}.

\begin{exmp}[stochastic Van der Pol oscillator] \label{ex:3}
Consider the following SDE:
\begin{equation}\label{eq4.1:new2}
d\left[\!\!
\begin{array}{c}
x_1(t) \\
x_2(t)
\end{array}\!\!
\right]= \left[\!\!
\begin{array}{c}
x_2(t) \\
\lambda\bigl[(1-x_1^2(t))x_2(t)-x_1(t)\bigr]
\end{array}\!\!
\right]dt +\left[\!\!
\begin{array}{cc}
0 & 0 \\
0 & 1
\end{array}\!\!
\right]dw(t)
\end{equation}
where the initial state $\bar x_0=[2,0]^\top$ and $\Pi_0={\rm diag} \{0.1, 0.1 \}$ with process covariance $Q=I_2$.

The measurement equation is taken to be
\begin{equation}\label{eq4.2:new2} z_k =[1, \: 1] x_k + v_k, \quad   \begin{array}{l}
    v_k \sim {\cal N}(0,R); \\
    R  = 0.04.
\end{array}
\end{equation}
\end{exmp}
For each value of the stiffness parameter $\lambda = 10^{0}, 10^{1}, \ldots, 10^{4}$, we perform the following set of numerical experiments. The SDE in~\eqref{eq4.1:new2} is simulated with a small stepsize $\delta_t = 10^{-5}$ on interval $[0, 2]$(s) to generate the true state vector. Next, equation~\eqref{eq4.2:new2} is utilized for creating the history of simulated measurements with the sampling rate $\Delta = 0.2$(s). Given the measurement data, the filtering problem is solved and the  hidden state is estimated by the filters under examination. We compute the accumulated root mean square error (ARMSE$_x$) in each entry of the state vector by averaging over $100$ Monte Carlo runs. Taking into account the potential stiffness of Example~3, we replace the conventional (that is, nonstiff) version of the NIRK method of order~6 utilized in the EKF-CKF estimators with its stiff version presented in~\cite{Ku20ZhVM,KuKu16SISCI,KuKu17MCS} in full detail. The obtained results are illustrated by Fig.~\ref{fig:3}.

\begin{figure}
\begin{tabular}{cc}
\includegraphics[width=0.5\textwidth]{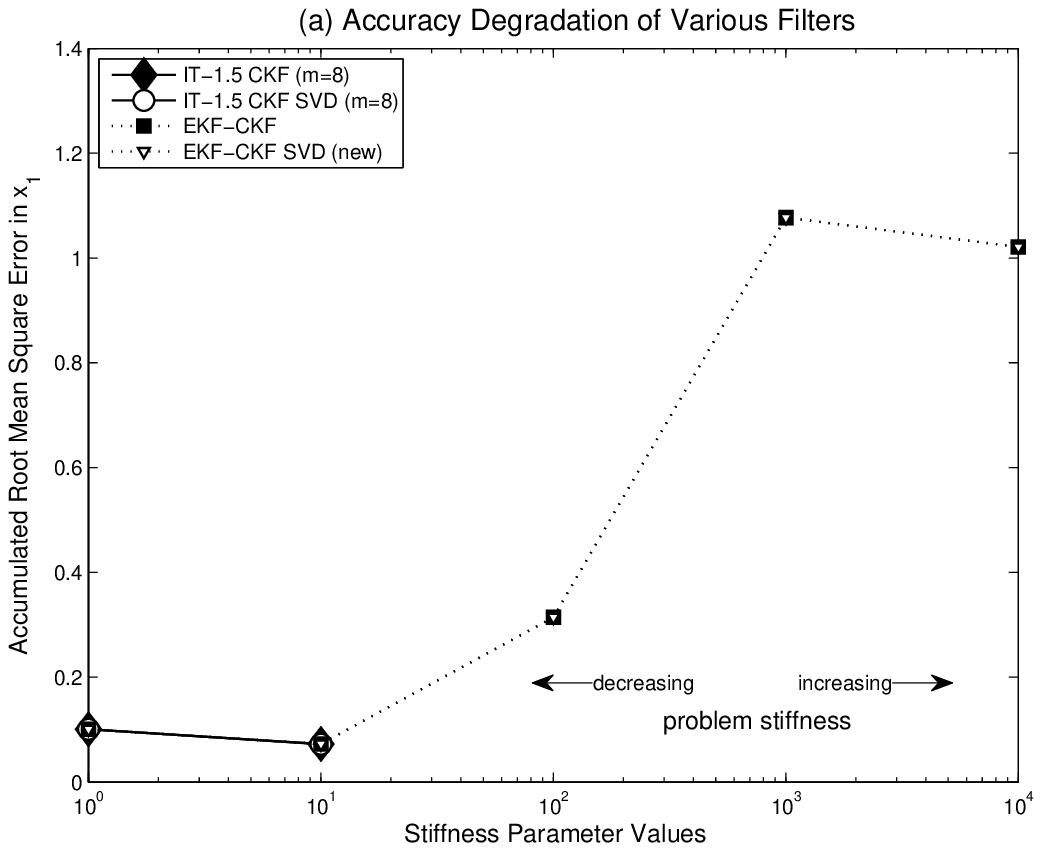} & \includegraphics[width=0.5\textwidth]{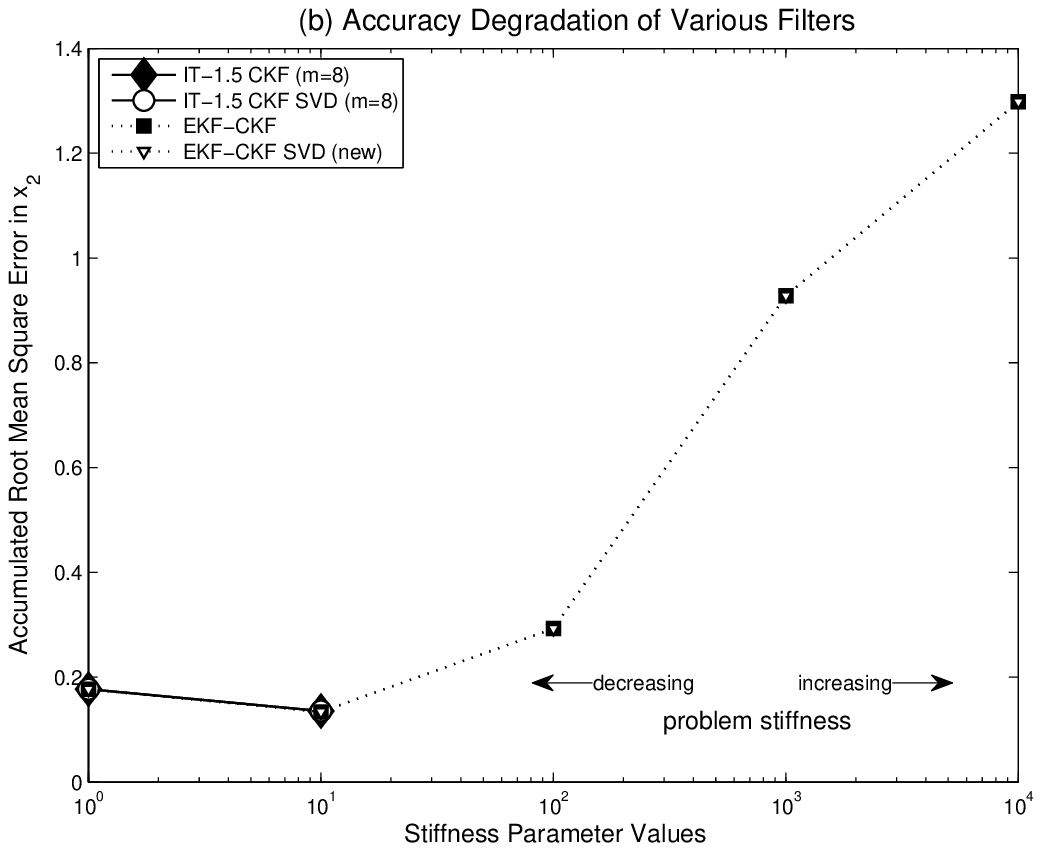}
\end{tabular}
\caption{The degradation of the filters' accuracy for various stiff parameter values in Example~\ref{ex:3}.} \label{fig:3}
\end{figure}

Having analyzed the obtained results, we conclude that both filtering strategies, i.e. the IT-1.5 CKF and NIRK-based hybrid EKF-CKF, solve the nonstiff estimation problems with the same accuracies. Besides, the estimation quality is high for both components of the hidden state vector, i.e. the estimation errors are small. However, the fixed stepsize CKF methods designed within the IT-1.5 numerical integration scheme  do not manage the stiff scenarios. As can be seen, the conventional IT-1.5 CKF and its SVD-based counterparts rapidly fail while the problem stiffness increases. More precisely, they fail for $\lambda > 10$, meanwhile the NIRK-based estimators are able to solve the stiff estimation problems till $\lambda = 10^{4}$. This creates a solid background for using the novel SVD variant of the NIRK-based hybrid EKF-CKF method for solving practical problems, including stiff scenarios as well as ill-conditioned cases.

\section{Concluding remarks} \label{conclusion}

In this paper, the mixed-type EKF-CKF filter for estimating hidden dynamic state of nonlinear stochastic systems is derived within the SVD-based spectral decomposition of the filter covariance matrices. In contrast to the previously obtained SVD-type CKF method based on the It\^{o}-Taylor expansion, the new estimator implies discretization error control while solving the related filters' moment differential equations. This yields a more accurate time update step due to the reduced discretization error and makes the new filtering method to be more robust and less sensitive to roundoff compared to the recently designed SVD-based CKF variant. Additionally, the adaptive nature of the novel estimator does not require any manual tuning prior to filtering, except the tolerance value to be given by users. The missing measurement scenario and/or irregular sampling interval cases are accurately solved in automatic mode.

There exists a space for further improvements by designing a more sophisticated CKF estimators. The solution implies a derivation of the CKF moment differential equations in terms of the SVD spectral factors of the covariance matrices involved. This allows for reducing the moment approximation error compared to the mixed-type EKF-CKF filter at the time update step meanwhile preserves all advanced features of the continuous-discrete approach related to discretization error control mentioned above, i.e. an automatic accurate and robust filtering way. However, it is presently unknown how to derive the SVD-based moment differential equations for the CKF estimation framework. 

Another important topic for a future research is a derivation of accurate nonlinear Bayesian filters by using alternative advanced ODE solvers with automatic stepsize selection and error control facilities, especially those grounded in Runge-Kutta, general linear and peer methods. We stress that the properties of the
Bayesian filters obtained heavily depend on the properties of the ODE solvers utilized for solving the related moment differential equations in the filtering scheme chosen for estimation. Thus, the development of a variety of the advanced filtering schemes and, then, investigation of their performance in comparative study when solving practical applications are also interesting topics for a future research.

\section*{Acknowledgments}
The authors acknowledge the financial support of the Portuguese FCT~--- \emph{Funda\c{c}\~ao para a Ci\^encia e a Tecnologia}, through the projects UIDB/04621/2020 and UIDP/04621/2020 of CEMAT/IST-ID, Center for Computational and Stochastic Mathematics, Instituto Superior T\'ecnico, University of Lisbon.

\section*{Appendix}

Here, we briefly note that system~\eqref{eq2.1} is solved by the variable-stepsize Gauss-type embedded NIRK pair of orders~4 and~6 with the built-in automatic combined local-global error control suggested in~\cite{2013:Kulikov:IMA}.  The details of the NIRK-based EKF-CKF estimator can be found in~\cite[Section~3]{2017:Kulikov:ANM}. Here, we briefly summarize the numerical scheme to be implemented on each sampling interval $[t_{k-1},t_k]$ of the EKF-CKF
as follows:
\begin{align}
\hat x_{l1}^2 & = a^2_{11} \hat x_l+a^2_{12} \hat x_{l+1} +\tau_l\bigl[d^2_{11}\!f(t_l,\hat x_l\!)\!+\!d^2_{12}f(t_{l+1},\hat x_{l+1}\!)\bigr], \nonumber \\
\hat  x_{l2}^2 & = a^2_{21} \hat x_l+a^2_{22} \hat x_{l+1} + \tau_l\bigl[d^2_{21}\!f(t_l,\hat x_l\!)\!+\!d^2_{22}f(t_{l+1},\hat x_{l+1}\!)\bigr], \nonumber \\
\hat x_{l1}^3 & = a^3_{11} \hat x_l+a^3_{12} \hat x_{l+1} + \tau_l\bigl[d^3_{11}\!f(t_l,\hat x_l\!)\!+\!d^3_{12}f(t_{l+1},\hat x_{l+1}\!) + d^3_{13}f(t_{l1}^2,\hat x_{l1}^2) + d^3_{14}f(t_{l2}^2,\hat x_{l2}^2)\bigr], \nonumber \\
\hat x_{l2}^3 & = a^3_{21} \hat x_l+a^3_{22} \hat x_{l+1} + \tau_l \bigl[d^3_{21}\!f(t_l,\hat x_l\!)\!+\!d^3_{22}f(t_{l+1},\hat x_{l+1}\!) + d^3_{23}f(t_{l1}^2,\hat x_{l1}^2) + d^3_{24}f(t_{l2}^2,\hat x_{l2}^2)\bigr], \nonumber \\
\hat x_{l3}^3 & = a^3_{31} \hat x_l+a^3_{32} \hat x_{l+1} +\tau_l\bigl[d^3_{31}\!f(t_l,\hat x_l\!)\!+\!d^3_{32}f(t_{l+1},\hat x_{l+1}\!)  + d^3_{33}f(t_{l1}^2,\hat x_{l1}^2) + d^3_{34}f(t_{l2}^2,\hat x_{l2}^2)\bigr], \nonumber \\
\hat x_{l+1} & = \hat x_{l}+\tau_l\bigl[b_1 f(t_{l1}^3,\hat x_{l1}^3)\!+\!b_2 f(t_{l2}^3,\hat x_{l2}^3) +b_3 f(t_{l3}^3,\hat x_{l3}^3)\bigr],\quad l=0,1,\ldots, end-1, \label{eq2.3}
\end{align}
where $\tau_l:=t_{l+1}-t_{l}$ is the step size of $\{t_l\}_{l=0}^{end}:=\left\{t_{l+1}=t_{l}+\tau_l, l=0,1,\ldots,end-1, t_0=t_{k-1},t_{end}=t_{k}\right\}$, and all constant coefficients of discretization~\eqref{eq2.3} are given; e.g., they are published
in~\cite{2016:Kulikov:SISCI}. The stage values $\hat x_{lj}^i:=\hat x(t_{lj}^i)$ mean approximations to the states evaluated at the time instants $t_{lj}^i=t_l+c_j^i\tau_l$. In summary, the discretized equation~\eqref{eq2.3} is iterated for an approximate state $\hat x_{l+1}$ at every node of the mesh $\{t_l\}_{l=0}^{end}$.

The {\em local error} ${le}_{l+1}$ associated with the predicted state mean vector ${\hat x}_{l+1}$ from (\ref{eq2.3}) is defined as follows:
\begin{equation}
le_{l+1} := \frac{\tau_l}{3}\Bigl[\frac23f(t_{l2}^3,\hat x_{l2}^{3})-\frac56f(t_{l1}^3,\hat x_{l1}^{3})-\frac56f(t_{l3}^3,\hat x_{l3}^{3}) + \frac12f(t_l,\hat x_{l})+\frac12f(t_{l+1},\hat x_{l+1})\Bigr].\label{eq2.4}
\end{equation}

The local error vector (\ref{eq2.4}) is measured in the scaled sense as
\begin{equation}\label{eq2.5}
|{le}_{l+1}|_{sc}:=\max_{i=1,2,\ldots,n}\bigl\{{|{le}_{i,l+1}|}/(|{\hat x}_{i,l+1}|+1)\bigr\}
\end{equation}
where the subscript~$i$ stands for the $i$-th entry in each
vector and $n$ implies the size of SDE~(\ref{eq1.1}). The magnitude $|{le}_{l+1}|_{sc}$ is referred to the {\em scaled local error} estimated at time $t_{l+1}$.

The {\em global} (or {\em true}) {\em error} of the ACD-EKF is evaluated on the discretization mesh~$\{t_l\}_{l=0}^{end}$ by the simple formula
\begin{equation}\label{eq2.6}
\Delta\hat x_{l+1}=\Delta\hat x_{l}-le_{l+1}
\end{equation}
where the vector $\Delta\hat x_{l+1}$ stands for the global error estimated at time $t_{l+1}$. The initial error $\Delta\hat x_{0}$ is always set to be zero. The global error (\ref{eq2.6}) is then assessed in the scaled sense
\begin{equation}\label{eq2.7}
|\Delta\hat x_{l+1}|_{sc}:=\max_{i=1,2,\ldots,n}\bigl\{{|\Delta\hat x_{i,l+1}|}/(|\hat
x_{i,l+1}|+1)\bigr\}
\end{equation}
where the first subscript~$i$ means the $i$-th entry in the corresponding vector. Magnitude in~\eqref{eq2.7} is referred to as the {\em scaled global error} estimated at a mesh node $t_{l+1}$.

Having computed the predicted state mean $\hat x_{l+1}$ at time $t_{l+1}$ with the original (or adjusted) stepsize $\tau_l$, for which the committed scaled local error $|\widetilde {le}_{l+1}|_{sc}$ does not exceed the bound $\epsilon_{loc}$, we solve the last set of the EKF moment differential equations in~\eqref{eq2.2}. For that, we use the same stepsize $\tau_l$ and the numerical scheme suggested in~\cite{2008:Mazzoni}
\begin{equation*}
P_{l+1}=M_{l+1/2} P_{l} M_{l+1/2}^{\top}+\tau_{l} K_{l+1/2} GQG^{\top} K_{l+1/2}^{\top}
\end{equation*}
where $t_{l+1/2}:=t_{l}+\tau_{l}/2$ is the mid-point of the $(l+1)$-st
step and
\begin{align*}
K_{l+1/2} & =\left[I_n - \frac{\tau_{l}}{2}F(t_{l+1/2}, \hat x_{l+1/2})\right]^{-1},  & M_{l+1/2} & = K_{l+1/2}\left[I_n + \frac{\tau_{l}}{2}F(t_{l+1/2}, \hat x_{l+1/2})\right].
\end{align*}


\section*{References}

\begin{thebibliography}{10}

\bibitem{abdi2021global}
A.~Abdi, G.~Hojjati, G.~Izzo, and Z.~Jackiewicz.
\newblock Global error estimation for explicit general linear methods.
\newblock {\em Numerical Algorithms}, pages 1--19, 2021.

\bibitem{2017:Arasaratnam}
I.~Arasaratnam and K.~P.~B. Chandra.
\newblock {12 Cubature Information Filters}.
\newblock {\em Multisensor Data Fusion: From Algorithms and Architectural
  Design to Applications}, page 193, 2017.

\bibitem{2009:Haykin}
I.~Arasaratnam and S.~Haykin.
\newblock {Cubature Kalman filters}.
\newblock {\em IEEE Transactions on Automatic Control}, 54(6):1254--1269, Jun.
  2009.

\bibitem{2010:Haykin}
I.~Arasaratnam, S.~Haykin, and T.~R. Hurd.
\newblock {Cubature Kalman filtering for continuous-discrete systems: Theory
  and simulations}.
\newblock {\em IEEE Transactions on Signal Processing}, 58(10):4977--4993, Oct.
  2010.

\bibitem{brehard2007hierarchical}
T.~Br{\'e}hard and J.-P. Le~Cadre.
\newblock Hierarchical particle filter for bearings-only tracking.
\newblock {\em IEEE Transactions on Aerospace and Electronic Systems},
  43(4):1567--1585, 2007.

\bibitem{Bu08}
J.~C. Butcher.
\newblock {\em Numerical Methods for Ordinary Differential Equations}.
\newblock John Wiley and Sons, Chichester, 2008.

\bibitem{2013:Chandra}
K.~P.~B. Chandra, D.-W. Gu, and I.~Postlethwaite.
\newblock {Square root cubature information filter}.
\newblock {\em IEEE Sensors Journal}, 13(2):750--758, 2013.

\bibitem{Co18}
E.~Constantinescu.
\newblock Generalizing global error estimation for ordinary differential
  equations by using coupled time-stepping methods.
\newblock {\em J. Comput. Appl. Math.}, 332:140--158, 2018.

\bibitem{1969:Dyer}
P.~Dyer and S.~McReynolds.
\newblock {Extensions of square root filtering to include process noise}.
\newblock {\em Journal of Optimization Theory and Applications}, 3(6):444--459,
  Jun. 1969.

\bibitem{2012:Frogerais}
P.~Frogerais, J.-J. Bellanger, and L.~Senhadji.
\newblock {Various ways to compute the continuous-discrete extended Kalman
  filter}.
\newblock {\em IEEE Transactions on Automatic Control}, 57(4):1000--1004, Apr.
  2012.

\bibitem{Frog2012}
P.~Frogerais, J.-J. Bellanger, and L.~Senhadji.
\newblock Various ways to compute the continuous-discrete extended
  \textsc\textit{{K}}alman filter.
\newblock {\em IEEE Trans. Automat. Contr.}, 57(4):1000--1004, Apr. 2012.

\bibitem{GoHe10}
S.~Gonz{\'a}lez-Pinto, D.~Hern{\'a}ndez-Abreu, and J.~I. Montijano.
\newblock An efficient family of strongly \textsc\textit{{$A$}}-stable
  \textsc\textit{{R}}unge-\textsc\textit{{K}}utta collocation methods for stiff
  systems and \textsc\textit{{DAE}}s.
  \textsc\textit{{P}}art~\textsc\textit{{I}}: Stability and order results.
\newblock {\em J. Comput. Appl. Math.}, 234:1105--1116, 2010.

\bibitem{GoHe12}
S.~Gonz{\'a}lez-Pinto, D.~Hern{\'a}ndez-Abreu, and J.~I. Montijano.
\newblock An efficient family of strongly \textsc\textit{{$A$}}-stable
  \textsc\textit{{R}}unge-\textsc\textit{{K}}utta collocation methods for stiff
  systems and \textsc\textit{{DAE}}s.
  \textsc\textit{{P}}art~\textsc\textit{{II}}: Convergence results.
\newblock {\em Appl. Numer. Math.}, 62:1349--1360, 2012.

\bibitem{2015:Grewal:book}
M.~S. Grewal and A.~P. Andrews.
\newblock {\em {Kalman Filtering: Theory and Practice using MATLAB}}.
\newblock John Wiley \& Sons, New Jersey, 4-th edition edition, 2015.

\bibitem{2010:Grewal}
M.~S. Grewal and J.~Kain.
\newblock {Kalman filter implementation with improved numerical properties}.
\newblock {\em IEEE Transactions on Automatic Control}, 55(9):2058--2068, Sep.
  2010.

\bibitem{HaWa96}
E.~Hairer and G.~Wanner.
\newblock {\em Solving Ordinary Differential Equations II: Stiff and
  Differential-Algebraic Problems}.
\newblock Springer-Verlag, Berlin, 1996.

\bibitem{1983:Ham}
F.~M. Ham and R.~G. Brown.
\newblock {Observability, eigenvalues, and Kalman filtering}.
\newblock {\em IEEE Transactions on Aerospace and Electronic Systems},
  (2):269--273, 1983.

\bibitem{Ja09}
Z.~Jackiewicz.
\newblock {\em General Linear Methods for Ordinary Differential Equations}.
\newblock John Wiley and Sons, Hoboken, 2009.

\bibitem{1970:Jazwinski:book}
A.~H. Jazwinski.
\newblock {\em Stochastic Processes and Filtering Theory}.
\newblock Academic Press, New York, 1970.

\bibitem{1999:Kloeden:book}
P.~E. Kloeden and E.~Platen.
\newblock {\em {Numerical Solution of Stochastic Differential Equations}}.
\newblock Springer, Berlin, 1999.

\bibitem{2013:Kulikov:IMA}
G.~Yu. Kulikov.
\newblock {Cheap global error estimation in some Runge--Kutta pairs}.
\newblock {\em IMA Journal of Numerical Analysis}, 33(1):136--163, 2013.

\bibitem{Ku20ZhVM}
G.~Yu. Kulikov.
\newblock Nested implicit \textsc\textit{{R}}unge-\textsc\textit{{K}}utta pairs
  of \textsc\textit{{G}}auss and \textsc\textit{{L}}obatto types with local and
  global error controls for stiff ordinary differential equations.
\newblock {\em Comput. Math. Math. Phys.}, 60(7):1134--1154, 2020.

\bibitem{2014:Kulikov:IEEE}
G.~Yu. Kulikov and M.~V. Kulikova.
\newblock {Accurate numerical implementation of the continuous-discrete
  extended Kalman filter}.
\newblock {\em IEEE Transactions on Automatic Control}, 59(1):273--279, 2014.

\bibitem{2016:Kulikov:SISCI}
G.~Yu. Kulikov and M.~V. Kulikova.
\newblock {Estimating the state in stiff continuous-time stochastic systems
  within extended Kalman filtering}.
\newblock {\em SIAM Journal on Scientific Computing}, 38(6):A3565--A3588, 2016.

\bibitem{KuKu16SISCI}
G.~Yu. Kulikov and M.~V. Kulikova.
\newblock Estimating the state in stiff continuous-time stochastic systems
  within extended \textsc\textit{{K}}alman filtering.
\newblock {\em SIAM J. Sci. Comput.}, 38(6):A3565--A3588, 2016.

\bibitem{KuKu17SP}
G.~Yu. Kulikov and M.~V. Kulikova.
\newblock Accurate continuous-discrete unscented \textsc\textit{{K}}alman
  filtering for estimation of nonlinear continuous-time stochastic models in
  radar tracking.
\newblock {\em Signal Process.}, 139:25--35, 2017.

\bibitem{2017:Kulikov:ANM}
G.~Yu. Kulikov and M.~V. Kulikova.
\newblock {Accurate cubature and extended Kalman filtering methods for
  estimating continuous-time nonlinear stochastic systems with discrete
  measurements}.
\newblock {\em Applied Numerical Mathematics}, 111:260--275, 2017.

\bibitem{2017:Kulikov:IEEE:mix}
G.~Yu. Kulikov and M.~V. Kulikova.
\newblock {Accurate state estimation in continuous-discrete stochastic
  state-space systems with nonlinear or nondifferentiable observations}.
\newblock {\em IEEE Transactions on Automatic Control}, 62(8):4243--4250, Aug.
  2017.

\bibitem{KuKu17MCS}
G.~Yu. Kulikov and M.~V. Kulikova.
\newblock Accurate state estimation of stiff continuous-time stochastic models
  in chemical and other engineering.
\newblock {\em Math. Comput. Simulation}, 142:62--81, 2017.

\bibitem{KuKu19IJRNC}
G.~Yu. Kulikov and M.~V. Kulikova.
\newblock Numerical robustness of extended \textsc\textit{{K}}alman filtering
  based state estimation in ill-conditioned continuous-discrete nonlinear
  stochastic chemical systems.
\newblock {\em Int. J. Robust Nonlinear Control}, 29(5):1377--1395, 2019.

\bibitem{KuKu20aANM}
G.~Yu. Kulikov and M.~V. Kulikova.
\newblock \textsc\textit{{NIRK}}-based \textsc\textit{{C}}holesky-factorized
  square-root accurate continuous-discrete unscented \textsc\textit{{K}}alman
  filters for state estimation in nonlinear continuous-time stochastic models
  with discrete measurements.
\newblock {\em Appl. Numer. Math.}, 147:196--221, 2020.

\bibitem{KuWe10SISCI}
G.~Yu. Kulikov and R.~Weiner.
\newblock Variable-stepsize interpolating explicit parallel peer methods with
  inherent global error control.
\newblock {\em SIAM J. Sci. Comput.}, 32(4):1695--1723, 2010.

\bibitem{KuWe11CAM}
G.~Yu. Kulikov and R.~Weiner.
\newblock Global error estimation and control in linearly-implicit parallel
  two-step peer \textsc\textit{{W}}-methods.
\newblock {\em J. Comput. Appl. Math.}, 236(6):1226--1239, 2011.

\bibitem{KuWe15SISCI}
G.~Yu. Kulikov and R.~Weiner.
\newblock A singly diagonally implicit two-step peer triple with global error
  control for stiff ordinary differential equations.
\newblock {\em SIAM J. Sci. Comput.}, 37(3):A1593--A1613, 2015.

\bibitem{KuWe20bANM}
G.~Yu. Kulikov and R.~Weiner.
\newblock Variable-stepsize doubly quasi-consistent singly diagonally implicit
  two-step peer pairs for solving stiff ordinary differential equations.
\newblock {\em Appl. Numer. Math.}, 154:223--242, 2020.

\bibitem{2020:Automatica:Kulikova}
M.~V. Kulikova and G.~Yu. Kulikov.
\newblock {SVD-based factored-form Cubature Kalman filtering for
  continuous-time stochastic systems with discrete measurements}.
\newblock {\em Automatica}, 120, 2020.
\newblock ~109110.

\bibitem{2017:Kulikova:IET}
M.~V. Kulikova and J.~V. Tsyganova.
\newblock {Improved discrete-time Kalman filtering within singular value
  decomposition}.
\newblock {\em IET Control Theory \& Applications}, 11(15):2412--2418, 2017.

\bibitem{2008:Mazzoni}
T.~Mazzoni.
\newblock {Computational aspects of continuous--discrete extended
  Kalman-filtering}.
\newblock {\em Computational Statistics}, 23(4):519--539, 2008.

\bibitem{Ma08}
T.~Mazzoni.
\newblock Computational aspects of continuous-discrete extended
  \textsc\textit{{K}}alman filtering.
\newblock {\em Comput. Statist.}, 23(4):519--539, 2008.

\bibitem{1988:Oshman}
Y.~Oshman.
\newblock {Square root information filtering using the covariance spectral
  decomposition}.
\newblock In {\em Proc. of the 27th IEEE Conf. on Decision and Control},
  volume~1, pages 382--387, 1988.

\bibitem{1986:Oshman}
Y.~Oshman and I.~Y. Bar-Itzhack.
\newblock {Square root filtering via covariance and information eigenfactors}.
\newblock {\em Automatica}, 22(5):599--604, 1986.

\bibitem{2018:Haykin}
E.~Santos-Diaz, S.~Haykin, and T.~R. Hurd.
\newblock {The fifth-degree continuous-discrete cubature Kalman filter radar}.
\newblock {\em IET Radar, Sonar \& Navigation}, 12(11):1225--1232, Nov. 2018.

\bibitem{2007:Sarkka}
S.~S\"arkk\"a.
\newblock {On unscented Kalman filter for state estimation of continuous-time
  nonlinear systems}.
\newblock {\em IEEE Transactions on Automatic Control}, 52(9):1631--1641, Sep.
  2007.

\bibitem{2012:Sarkka}
S.~S\"arkk\"a and A.~Solin.
\newblock {On continuous-discrete cubature Kalman filtering}.
\newblock {\em IFAC Proceedings Volumes}, 45(16):1221--1226, 2012.

\bibitem{ScWe04}
B.~A. Schmitt and R.~Weiner.
\newblock Parallel two-step \textsc\textit{{W}}-methods with peer variables.
\newblock {\em SIAM J. Numer. Anal.}, 42:265--286, 2004.

\bibitem{ScWe05BIT}
B.~A. Schmitt, R.~Weiner, and H.~Podhaisky.
\newblock Multi-implicit peer two-step \textsc\textit{{W}}-methods for parallel
  time integration.
\newblock {\em BIT}, 45:197--217, 2005.

\bibitem{2017:Tsyganova:IEEE}
J.~V. Tsyganova and M.~V. Kulikova.
\newblock {SVD-based Kalman filter derivative computation}.
\newblock {\em IEEE Transactions on Automatic Control}, 62(9):4869--4875, Sep.
  2017.

\bibitem{1992:Wang}
L.~Wang, G.~Libert, and P.~Manneback.
\newblock {Kalman filter algorithm based on Singular Value Decomposition}.
\newblock In {\em Proceedings of the 31st Conference on Decision and Control},
  pages 1224--1229, Tuczon, AZ, USA, Dec. 1992.

\bibitem{WeKu14CAM}
R.~Weiner and G.~Yu. Kulikov.
\newblock Local and global error estimation and control within explicit
  two-step peer triples.
\newblock {\em J. Comput. Appl. Math.}, 262:261--270, 2014.

\bibitem{WeKu17CAM}
R.~Weiner, G.~Yu. Kulikov, S.~Beck, and J.~Bruder.
\newblock New third- and fourth-order singly diagonally implicit two-step peer
  triples with local and global error controls for solving stiff ordinary
  differential equations.
\newblock {\em J. Comput. Appl. Math.}, 316:380--391, 2017.

\bibitem{WeKu12ANM}
R.~Weiner, G.~Yu. Kulikov, and H.~Podhaisky.
\newblock Variable-stepsize doubly quasi-consistent parallel explicit peer
  methods with global error control.
\newblock {\em Appl. Numer. Math.}, 62(10):1591--1603, 2012.

\bibitem{WeSc09}
R.~Weiner, B.~A. Schmitt, H.~Podhaisky, and S.~Jebens.
\newblock Superconvergent explicit two-step peer methods.
\newblock {\em J. Comput. Appl. Math.}, 223:753--764, 2009.

\end{thebibliography}

\end{document}